\documentclass[a4paper,11pt]{article}
\usepackage[margin=1in]{geometry}
\usepackage{amsmath,amssymb,amsthm}
\usepackage{mathtools}
\usepackage{graphicx}
\usepackage[authoryear]{natbib}
\usepackage{enumitem}
\usepackage{setspace}
\usepackage{subfigure}
\usepackage{bm}
\usepackage{authblk}
\RequirePackage[colorlinks,citecolor=blue,urlcolor=blue]{hyperref}
\newcommand{\R}{\mathbb{R}}

\newcommand{\E}{\mathbb{E}}
\newcommand{\AP}{\text{AP}}

\providecommand{\abs}[1]{\lvert#1\rvert}

\DeclareMathOperator*{\argmax}{argmax}
\newtheorem{theorem}{Theorem}[section]
\newtheorem{lemma}{Lemma}[section]
\newtheorem{proposition}{Proposition}[section]

\newtheorem{example}{Example}
\newtheorem{remark}{Remark}[section]
\theoremstyle{definition}
\numberwithin{equation}{section}

\newtheorem{assumption}{Assumption}

\begin{document}

\title{\bf  Joint pricing and inventory control for a stochastic inventory system with Brownian motion demand}

\author{Dacheng Yao}
\affil{\small Academy of Mathematics and Systems Science, Chinese
Academy of Sciences, Beijing, 100190, China; dachengyao@amss.ac.cn} 
\date{}
\maketitle
\begin{abstract}
In this paper, we consider an infinite horizon, continuous-review, stochastic inventory system in which
cumulative customers' demand is price-dependent and is modeled as a Brownian motion.
Excess demand is backlogged.
The revenue is earned by selling products and
the costs are incurred by holding/shortage and ordering,
the latter consists of a fixed cost and a proportional cost.
Our objective is to simultaneously determine a pricing strategy and an inventory control strategy
to maximize the expected long-run average profit.
Specifically, the pricing strategy provides the price $p_t$ for any time $t\geq0$
and the inventory control strategy characterizes when and how much we need to order.
We show that an $(s^*,S^*,\bm{p}^*)$ policy is optimal and obtain the equations of optimal
policy parameters, where $\bm{p}^*=\{p_t^*:t\geq0\}$.
Furthermore, we find that at each time $t$, the optimal price $p_t^*$ depends on the current inventory level $z$,
and it is increasing in $[s^*,z^*]$ and is decreasing in $[z^*,\infty)$, where $z^*$
is a negative level.

\vspace{.25in}

\end{abstract}

\section{Introduction}
\label{}

Exogenous selling price is always assumed in the classic production/inventory models;
see e.g., \cite{Scarf1960} and the research thereafter.
In practice, however, \emph{dynamic} pricing
is one important tool for revenue management by balancing customers' demand and inventory level.
For example, airline firms change prices timely according to the number of unsold tickets
(see \cite{ChaoZhou2006}),
and Dell's price reflects real-time demand fluctuations, and varies significantly from week to week
(see \cite{Byrnes2003}).
Dynamic pricing is usually coordinated with inventory control in many firms,
e.g., Dell, Amazon, FairMarket, MSN Auction, etc. (see \cite{Feng2010}).
In this paper, we will consider an optimization problem of combining dynamic pricing and inventory control
in a continuous-review stochastic inventory model,
and analyze the dependence of price on the inventory level.

The topic of joint pricing and inventory control has been studied in different models,
and most focus on periodic-review inventory models, in which a \emph{fixed} price is chosen for each period.
As the first work on this topic, \cite{Whitin1955} studies it in a newsboy problem; see
a recent survey in the newsboy setting in \cite{PetruziDada1999}.
 \cite{FedergruenHeching1999} studies a periodic-review stochastic inventory model
without any fixed ordering cost and shows that a base-stock list price policy is optimal.
Later, \cite{ChenSimchi-Levi2004a,ChenSimchi-Levi2004b} study periodic-review stochastic inventory models
with a fixed ordering cost in finite horizon and infinite horizon respectively.
They show that the optimal policies are of $(s,S,\bm{p})$-type, under which,
the inventory control in each period is the classical $(s,S)$ policy
and the price is determined based on the inventory position at the beginning of each period.
Currently, there are many extensions for joint pricing and inventory control in
periodic-review stochastic inventory models,
e.g., \cite{YinRajaram2007} for a problem with Markovian demand,
\cite{ChaoChenZheng2008} and \cite{Feng2010}
for problems with supply capacity,
\cite{ChungTalluriNarasimhan2015} for a problem with multiple price markdowns,
and \cite{BernsteinLiShang2016} for a problem with ordering leadtime.
Furthermore, dynamic pricing has been concerned in few continuous-review inventory models
where demand arrives randomly at discrete time.
\cite{ChenSimchi-Levi2006} studies an inventory model with a general semi-Markov process demand,
where customer's demand arrives randomly at discrete time with independent interval time
and demand size is a random variable and independent of each other,
and proves the optimality of an $(s,S,\bm{p})$ policy using dynamic programming.
Based on this optimality of an $(s,S,\bm{p})$ policy,
\cite{ChaoZhou2006} focuses on the $(s,S,\bm{p})$ policies in a special model with Poisson demand,
and shows that the price is a unimodal function of the inventory level with maximum point at zero.

In this paper, we study a continuous-review stochastic inventory system, where the cumulative demand
is modeled as a Brownian motion with positive price-dependent drift rate.
Excess demand is backlogged.
The revenue is earned by selling products and the costs are incurred by holding/shortage and ordering,
the latter consists of a fixed cost and a unit cost.
Our objective is to find a policy that jointly optimizes dynamic pricing and ordering strategies
to maximize the long-run average profit. Specifically,
the dynamic pricing strategy lists the price for all time $t\geq0$.
Using an upper bound approach,
we prove that an $(s^*,S^*,\bm{p}^*)$ policy is optimal for our model,
and we further show that the optimal price $p_t^*$ at any time $t$ completely depends on the current inventory level $z$,
and there exists a $z^*\leq0$ such that $p_t^*$ is increasing in $[s^*,z^*]$ and is decreasing in $[z^*,\infty)$.

There are two reasons that motivate us to model the cumulative demand process as a Brownian motion with strictly positive drift rate:
First, a demand with normal distribution is usually used to approximate customer demand in the periodic-review
stochastic inventory system literature (see e.g., \cite{Porteus2002} and \cite{Zipkin2000}),
and Brownian approximation is constructed to study many production-inventory systems
(see e.g., \cite{AllonVanMieghem2010}, \cite{Bradley2004} and \cite{BradleyGlynn2002}).
In addition, Brownian demand can capture customers' return and has been used in \cite{Bather1966}, \cite{Gallego1990}, etc.
Second, Brownian motion brings tractability for the analysis, and provides specific properties for the optimal pricing.
In particular, Brownian control problem is tractable for proving the optimality of selected policy
by using our upper bound approach.
Also, the possibility of negative demand (i.e. customers' return) results in a negative turnover point $z^*$
such that the optimal price at any time is increasing on the inventory level
when the level is in $[s^*,z^*]$ and is decreasing when the level in $[z^*,\infty)$;
see more discussion after Theorem \ref{thm:mainresults}.

The inventory control models with Brownian motion demand have been studied extensively,
starting from \cite{Bather1966}. This pioneering work studies an impulse control
of a Brownian inventory model under average cost criterion with a fixed ordering cost,
and shows that an $(s,S)$ policy is optimal.
It is then extended to several different models.
\cite{Sulem1986} and \cite{Benkherouf2007} consider Brownian inventory models
under discounted cost criterion
with linear and general holding/shortage cost functions respectively;
\cite{Bar-IlanSulem1995} and \cite{MuthuramanSeshadriWu2015} consider Brownian inventory models
with constant lead times and stochastic lead times respectively;
and \cite{YaoChaoWu2015} and \cite{HeYaoZhang2016} consider Brownian inventory models
with concave ordering costs and quantity-dependent set-up costs respectively.
They all prove the optimality of $(s,S)$ policies for the models considered.
It's worth mentioning that the upper bound approach adopted in this paper is similar to the lower bound approach in
\cite{YaoChaoWu2015} and \cite{HeYaoZhang2016}.
However, the analysis for the structure of the relative value function
(it is the key to characterize the optimal policy parameters) in our paper is very different.
Specifically, it is easy to analyze the properties of the relative value function in their works
by obtaining its explicit solution.
However, in this paper, it is difficult to get an explicit solution from the ordinary differential equation (ODE)
since the drift rate depends on the state (see \eqref{eq:w-ODE}).
To overcome this, we prove an important proposition (see Proposition \ref{prop:w}) to analyze the properties from the ODE and boundary conditions directly.
This method doesn't require an explicit solution
and can be applied to more general models (e.g., more general demand process models).

The most related to our work are \cite{ChenWuYao2010} and \cite{ZhangZhang2012},
which, under no backlog assumption,
study joint pricing and inventory control in some special Brownian inventory systems.
\cite{ChenWuYao2010} studies a Brownian inventory model,
in which $(0,S)$-type inventory control strategies are assumed and
a constant price $p_n$ is implemented when the inventory level is in $(S_n,S_{n-1}]$
with $S=S_0>S_1>\ldots>S_N=0$ and $S_n=(N-n)S/N$, $n=0,1,\ldots,N$ ($N$ is fixed).
The optimal policy parameters $(S,p_1,p_2,\ldots,p_N)$ are identified for this
discrete price adjustment system with a linear holding cost.
\cite{ZhangZhang2012} further considers a special Brownian inventory model,
where demand rate as a function of price $p\in[a,b]$ is given by
$\lambda(p)=\lambda_1/(\lambda_0+p)$ for positive constants $\lambda_0$ and $\lambda_1$.
With this special drift rate, the optimal price is discrete in the inventory level;
i.e. there exists a positive inventory level $x^*$ such that for each time $t$,
the optimal price $p_t^*=b$ for $z\leq x^*$ and $p_t^*=a$ for $z>x^*$.
In short, these two papers just solve some special no backlog Brownian inventory models
with discrete price adjustments,
and it is still open if $(s,S,\bm{p})$ policy is optimal
for the general Brownian inventory model with backlog.
Our objective in this paper is to prove the optimality of the $(s,S,\bm{p})$ policy
where backorders are allowed,
and to give general price adjustments including both discrete and continuous adjustments.

Since the demand drift rate depends on the price,
a stream of literature related to our work is drift rate control problem.
\cite{AtaHarrisonShepp2005} studies a Brownian processing problem
in which manager can continuously modify the drift rate
when the inventory level is limited in a finite interval $[0,b]$.
Under some assumptions for control cost,
the authors explicitly solve the problem and characterize the equations of optimal cost and
drift rate.
Later, \cite{OrmeciVandeVate2011} studies a Brownian control problem
where the drift rate can be selected in a finite set and
changeover costs are charged when the controller changes the drift rate.
The authors demonstrate that the problem admits an optimal policy that is
a deterministic non-overlapping control band policy.
For more works about drift rate control, we refer to
\cite{AtaHarrisonShepp2005}, \cite{OrmeciVandeVate2011} and references therein.

The contribution of this paper can be summarized as follows.
First, although there is an extensive literature studying the joint pricing and inventory control,
most focus on the periodic-review models. In this paper,
we prove the optimality of $(s,S,\bm{p})$ policy
for a continuous-review inventory model with Brownian motion demand
and explicitly characterize the optimal policy parameters by an ODE
with some boundary conditions.
This fills in the gap on jointly optimizing dynamic pricing and inventory control for
a general Brownian inventory system allowing backlog.
Second, we find that the optimal price depends on the current inventory level
and is its a general function, including both discrete and continuous cases.
We further prove that the optimal price first increases
and then decreases as the inventory level decreases from $S^*$ to $s^*$ with a negative turnover point $z^*$,
which is different to the Poisson demand model (\cite{ChaoZhou2006}) with zero as the turnover point.
Third, we provide an approach to characterize the properties of the derivative of relative value function and the optimal policy parameters
using the ODE and the boundary conditions directly.
Comparing with the lower bound approach in literature (see e.g., \cite{YaoChaoWu2015} and \cite{HeYaoZhang2016}),
the approach in this paper doesn't require an explicit solution for the relative value function
and thus could be applied to the models with more general demand process.

The rest of this paper is as follows. In Section \ref{sec:modelandmainresults},
we model our problem and state the main results.
In Section \ref{sec:upperbound}, we establish an upper bound for the profit under any admissible policy.
Then in Section \ref{sec:pricing}, we analyze the optimal pricing strategy under
a given $(s,S)$ inventory control strategy.
In Section \ref{sec:inventorycontrol}, we devote to an optimal inventory control strategy
by using the upper bound established in Section \ref{sec:upperbound}.
Finally, we conclude this paper in Section \ref{sec:conclusion}.
We close this section with some frequently used notation.
Let $\mathcal{C}^1(\R)$ denote the set of functions with continuous derivatives on $\R$.
For a real-valued function $f$ on $[0,\infty)$,
let $f_{t-}$ ($f_{t+}$) be its left (right) limit at $t$ and $\Delta f_t=f_{t+}-f_{t-}$.
Let $1_A$ be the indicator function of $A$, i.e., $1_A=1$ if $A$ occurs and $1_A=0$ otherwise.

\section{Model and main results}
\label{sec:modelandmainresults}

\subsection{Model}

Consider a continuous-review inventory system, whose cumulative customers' demand up to time $t$
is given by
\begin{equation}
\label{eq:D}
D_t=\int_0^t\mu(p_u)\,\mathrm{d}u+\sigma B_t,\quad t\geq0,
\end{equation}
where $\mu(p_t)$ is the demand rate at time $t$ and depends on the product price
$p_t\in[\underline{p},\bar{p}]$, $\bar{p}>\underline{p}\geq0$,
$\sigma^2$ is the variance and is a constant,
and $B=\{B_t:t\geq0\}$ denotes a standard Brownian motion with $B_0=0$.
We assume $B$ is defined on some filtered probability space
$(\Omega,\{\mathcal{F}_t\},\mathcal{F},\mathbb{P})$ and $B$ is an $\mathcal{F}_t$-martingale.
The demand process can be understood as the \textit{netput} process,
which denotes the difference between the real demand and some input processes,
e.g., customer returns.
Notice that the variance (i.e. the uncertainty) in the demand \eqref{eq:D} is assumed to be independent to the price.
It is very similar to the additive demand in the periodic-review model in literature;
see e.g., \cite{ChenSimchi-Levi2004a} and \cite{Feng2010}.
This assumption is common in economic literature (e.g., \cite{Mills1959})
and this kind of demand is reasonable to model branded products (e.g., Intel's processors)
that exhibit substantial customer demand; see \cite{LiZheng2006}.

Let $Z_t$ denote the controlled inventory level at time $t\geq0$.
In this system, there are two control strategies: an inventory control strategy $\bm{y}=\{y_t:t\geq0\}$
and a pricing strategy $\bm{p}=\{p_t:t\geq0\}$,
where $y_t$ denotes the cumulative order quantity up to time $t$
and $p_t$ denotes the product price at time $t$.
More specifically, an inventory control strategy $\bm{y}$ can be
expressed as a sequence $\{(\tau_i,\xi_i):i=0,1,2,\cdots\}$,
where $\tau_i$ denotes the $i^{th}$ ordering time and $\xi_i$ denotes the ordering quantity at time $\tau_i$.
We let $\tau_0=0$ and $\xi_0\geq0$ by convention.
Then,
\[
y_t=\sum_{i=0}^{N_t}\xi_i,\quad t\geq0,
\]
where $N_t=\sup\{i:\tau_i\leq t\}$ denotes the numbers of order up to time $t$ except possible order at time $0$.
An inventory control strategy $\bm{y}=\{y_t:t\geq0\}$ is called \emph{admissible}
if it satisfies the following conditions (similar conditions can be found in \cite{YaoChaoWu2015}):
\begin{itemize}
\item[(i)] $\bm{y}$ is adapted to $\mathcal{F}$, i.e.,
$\tau_i$ is an $\mathcal{F}$-stopping time,  $\xi_i$ is $\mathcal{F}_{\tau_i-}$-measurable.
\item[(ii)] For each sample path $\omega\in\Omega$, $\bm{y}(\omega,\cdot)$ is RCLL (right continuous on $[0,\infty)$ and has left limits on $(0,\infty)$).
\item[(iii)] For each inventory control strategy $\bm{y}$, there exists some finite number $C_{\bm{y}}$ such that $\bm{y}$ does not place an order when the inventory level $Z$ is greater than $C_{\bm{y}}$.
Because $C_{\bm{y}}$ can be arbitrarily large, all policies of practical interest are contained.
\end{itemize}
Under an admissible inventory control strategy $\bm{y}$,
the controlled inventory level $Z$ is given by
\begin{equation}
\label{eq:Z}
Z_t=x-D_t+y_t=x-\int_0^t\mu(p_u)\,\mathrm{d}u-\sigma B_t+\sum_{i=0}^{N_t}\xi_i, \quad t\geq0,
\end{equation}
where $Z_{0-}=x$ is the initial inventory level.
Furthermore, a  pricing strategy $\bm{p}=\{p_t:p_t\in[\underline{p},\bar{p}], t\geq0\}$
is {called \emph{admissible} if $\bm{p}$ is adapted to $\mathcal{F}$.
Let $\phi=(\bm{y},\bm{p})$ be an admissible policy including an admissible inventory control strategy $\bm{y}$
and an admissible pricing strategy $\bm{p}$, and let $\Phi$ be the set of all admissible policies.

We now introduce the revenue and costs in our system.
For the revenue, we assume that customers pay the fees upon arrival when backorder happens.
This assumption is common in the existing pricing literature where backorders are allowed;
see e.g., \cite{ChenSimchi-Levi2004a} and \cite{BernsteinLiShang2016}.
Thus, the revenue for our system up to time $t$ is given by
\begin{equation}
\label{eq:revenue}
\int_0^t p_u\,\mathrm{d}D_u.
\end{equation}
Using \eqref{eq:D}, the expectation of \eqref{eq:revenue} can be written as
\[
\E_x\Big[\int_0^t p_u\,\mathrm{d}D_u\Big]
=\E_x\Big[\int_0^t p_u\mu(p_u)\,\mathrm{d}u\Big],
\]
where $\E_x$ denotes the expectation operator conditioning on the initial inventory level $Z_{0-}=x$.
Furthermore, there are two costs in our system: holding/shortage cost and ordering cost.
The holding/shortage cost is charged at rate $h(z)$ when the inventory level
is $z$, i.e., $h(z)$ is the holding cost per unit of time
when $z\geq0$ and is the shortage cost per unit of time when $z<0$.
Thus, the expected holding/shortage cost in $[0,t]$ is
\[
\E_x\Big[\int_0^t h(Z_u)\,\mathrm{d}u\Big].
\]
The ordering cost is given by a linear function $K+k\xi$ for each order quantity $\xi>0$,
where $K>0$ is the setup cost and $k>0$ is the unit cost.

Under an admissible policy $\phi\in\Phi$, the long-run average profit is given by
\begin{eqnarray*}
{\sf AP}(x,\phi)
=\liminf_{t\to\infty}\frac{1}{t}\E_x \Big[\int_0^t \big[p_u\mu(p_u)-h(Z_u)\big]\,\mathrm{d}u
-\sum_{i=0}^{N_t}\big(K+k\xi_i\big)\Big],\label{eq:AP}
\end{eqnarray*}
and our objective is to find a $\phi^*\in\Phi$ solving
\begin{equation}
\label{eq:supAP}
{\sf AP}(x,\phi^*)=\sup_{\phi\in\Phi}{\sf AP}(x,\phi).
\end{equation}
Finally, we present the assumptions in this paper as follows.
Assumption \ref{assumption-mu} is for the demand rate function $\mu$,
and Assumption \ref{assump:h-qd} gives some usual assumptions for the holding/shortage cost function $h$;
see \cite{WuChao2014} for similar assumptions about $h$.
\begin{assumption}
\label{assumption-mu}
The demand rate $\mu(p):[\underline{p},\overline{p}]\to(0,\infty)$ satisfies the following assumptions.
\begin{itemize}
\item[(a)] $\mu(p)$ is twice continuously differential; and
\item[(b)] $\mu(p)$ is strictly decreasing, i.e. $\mu'(p)<0$.
\end{itemize}
\end{assumption}
\begin{assumption}
\label{assump:h-qd}
The holding/shortage cost function $h: \R\to[0,\infty)$ satisfies the following assumptions:
\begin{itemize}
\item[(a)] $h$ has a minimum at $z=0$ and $h(0)=0$;
\item[(b)] $h$ is continuously differentiable except at $z=0$;
\item[(c)] $h$ is strictly convex; and
\item[(d)] $h$ is polynomially bounded, i.e., there exist constants $c_1\in\R$, $c_2>0$, and an integer $n\in\mathbb{N}^+$
such that $h(z)\leq c_1+c_2 |z|^n$ for all $z\in \R$.
\end{itemize}
\end{assumption}
It follows from Assumption \ref{assump:h-qd} (a)-(c) that we have
\begin{eqnarray}
&& h'(z)
\left\{
\begin{array}{ll}
<0 & \text{for $z<0$},\\
>0 & \text{for $z>0$},
\end{array}
\right.\label{eq:h'}\\
&&\lim_{|z|\to \infty} h(z) = +\infty,\label{eq:limh}
\end{eqnarray}
and there exist constants $d_1>0$ and $d_2$ such that
\begin{equation}
\label{eq:h-p-q}
h(z)\geq d_1\abs{z}+d_2.
\end{equation}

\subsection{Main results}

To present our main results, we first introduce one class of inventory control strategies
and one class of pricing strategies.

For the inventory control, we aim to find an $(s,S)$ policy to
be optimal. Under an $(s,S)$ policy,
the system orders items immediately to $S$
once the inventory level is lower than or equals to $s$.
More specifically, $(s,S)$ policy can be given by
\begin{equation}
\label{eq:sS-0}
\tau_0=0,\quad
\xi_0=\left\{
\begin{array}{ll}
S-x & \text{if $Z_{0-}=x\leq s$},\\
0 & \text{if $Z_{0-}=x>s$},
\end{array}
\right.
\end{equation}
and for $i\geq1$,
\begin{equation}
\label{eq:sS-i}
\tau_i=\inf\{t> \tau_{i-1}:Z_{t-}=s\},
\quad
\xi_i=S-s.
\end{equation}
We provide one class of pricing strategies as follows.
Define
\begin{align}
\label{eq:Pi}
\Pi(p,w)=\mu(p)[p-w]\quad \text{for any $p\in[\underline{p},\bar{p}]$ and $w\in\R$},
\end{align}
and let $p_{\pi}(w)$ be the smallest one in $[\underline{p},\bar{p}]$
that maximizes $\Pi(p,w)$ for given $w\in\R$, i.e.,
\begin{align}
\label{eq:p-pi}
p_{\pi}(w)=\min\argmax_{p\in[\underline{p},\bar{p}]} \Pi(p,w).
\end{align}
Thus, for each given function $w(\cdot):\R\to\R$,
we define a pricing strategy as
\begin{equation}
\label{eq:pricing}
\bm{p}=\{p_t=p_{\pi}(w(Z_t)):t\geq0\}.
\end{equation}
In the following main results, $w(\cdot)$ is actually the derivative of relative value function $V$, i.e., $w=V'$.
The detailed explanation can be found after Theorem \ref{thm:mainresults}.

Theorem \ref{thm:mainresults} below states our main results on the optimal policy characterization.
\begin{theorem}
\label{thm:mainresults}
Under Assumption \ref{assumption-mu} and \ref{assump:h-qd},
if there exists $((s^*,S^*),w^*(\cdot),\gamma^*)$,
where $s^*$, $S^*$ and $\gamma^*$ are constants and
$w^*(\cdot)$ is continuously differentiable in $\R$ and twice continuously differentiable in $\R\setminus \{0\}$,
satisfying the following ODE
\begin{eqnarray}
\label{eq:w-ODE}
\frac{\sigma^2}{2}w'(z)+\pi(w(z))-h(z)=\gamma \quad\text{for $z\in\R$}
\end{eqnarray}
with boundary conditions
\begin{eqnarray}
&&\int_{s}^S \big[w(z)-k\big]\mathrm{d}z=K,\label{eq:w-int=K}\\
&&w(s)=w(S)=k,\label{eq:w(s)=w(s)=k}\\
&&\lim_{z\to\infty}w'(z)/h(z)=0, \label{eq:limw'/h}
\end{eqnarray}
where $s<S$ and
\begin{align}
\label{eq:pi}
\pi(w)=\Pi(p_{\pi}(w),w)=\mu(p_{\pi}(w))[p_{\pi}(w)-w],
\end{align}
then we have the following.
\begin{itemize}
\item[(1)]
$((s^*,S^*),\bm{p}^*)$ policy with $\bm{p}^*=\{p_t=p_{\pi}(w^*(Z_t)):t\geq 0\}$
is an optimal policy and $\gamma^*$ is the optimal profit for our problem,
i.e., $(s^*,S^*)$ policy is an optimal inventory control strategy
and $\bm{p}^*$ is the associated optimal pricing strategy.
\item[(2)] There exists a unique $z^*$ with $z^*\leq0$ and $s^*<z^*<S^*$
such that $p^*(z)\triangleq p_{\pi}(w^*(z))$ is increasing in $[s^*,z^*]$,
and decreasing in $[z^*,\infty)$.
\end{itemize}
\end{theorem}
\begin{remark}
This paper proves the optimality of $(s^*,S^*,\bm{p}^*)$ policy by
assuming the existence of solution for \eqref{eq:w-ODE}-\eqref{eq:limw'/h},
and the same assumptions can be found in \cite{Benkherouf2007} and \cite{LaknerReed2010}.
Indeed, we can get the explicit solutions when $\mu(p)$ has special forms,
e.g., the case in the following Example 1.
\end{remark}
\begin{remark}
Under $(s^*,S^*)$ policy, if the initial inventory level
is lower than $s^*$, i.e. $Z_{0-}=x<s^*$, an order
will be placed immediately to increase the inventory level
up to $S^*$, and then the inventory level will not be lower than $s^*$
because of the continuous demand. Furthermore,
since the demand process is a Brownian motion, the inventory level
is possibly higher than $S^*$.
Thus, the controlled inventory level satisfies
\[
Z_t\geq s^*, \quad t\geq 0.
\]
This is the reason why we need to give the monotonicity of $p^*(z)$ on $[s^*,\infty)$.
\end{remark}

First, we provide an explanation for \eqref{eq:w-ODE}. Under any given $(s,S)$ policy and any pricing strategy $\bm{p}=\{p_t:t\geq0\}$,
let $X_t=z-\int_0^t\mu(p_u) \,\mathrm{d}u-\sigma B_t$ be the inventory level at time $t$ when the initial inventory level is $z>s$ and no order is placed in $[0,t]$,
and let
\[
V(z)=\gamma \mathbb{E}_x\big[\tau(s)\big]-\mathbb{E}_x\Big[\int_0^{\tau(s)} \big(p_u\mu(p_u)-h(X_u)\big)\,\mathrm{d} u\Big],
\]
where $\tau(s)$ denotes the first ordering time, i.e., the first time when the inventory hits level $s$.
$V(z)$ is called \emph{relative value function} in literature
(see e.g., \cite{OrmeciVandeVate2011}) and can be interpreted as the expected benefit advantage of reorder level $s$ relative to the inventory level $z$.
It is known that $V$ satisfies the following ordinary differential equation
(see e.g., pp. 170 in \cite{WuChao2014})
\begin{equation}
\frac{\sigma^2}{2} V''(z)-\mu(p_0)V'(z)+p_0\mu(p_0)-h(z)=\gamma.
\end{equation}
Let $w(z)=V'(z)$ and choose $\bm{p}$ following \eqref{eq:pricing},
then $w$ satisfies \eqref{eq:w-ODE}.

We next provide an intuitive argument on the conditions \eqref{eq:w-int=K}-\eqref{eq:limw'/h}
that the relative value function and the optimal policy parameters should satisfy.
First, under $(s,S)$ policy,
an order with quantity $S-s$ will incur cost $K+k(S-s)$, and it follows from the definition of the value function $V$ that
\[
V(s)=V(S)-K-k(S-s),
\]
which yields the condition \eqref{eq:w-int=K} with $V(z)=\int_s^z w(y)\,\mathrm{d}y$.
Also, it is known that when $h$ is polynomial with order $n$,
$V(z)$ has order $n+1$ and $w'(z)=V''(z)$ has order $n-1$ as $z$ goes to infinity (see Section 4 in \cite{WuChao2014}).
Thus, \eqref{eq:limw'/h} holds.
Next, we show the conditions that should be imposed on the optimal parameters $s^*$ and $S^*$.
Starting from $s$, the system should jump to an $S$ that maximizes
\[
V(S)-K-k(S-s).
\]
Therefore, $\mathrm{d}(V(S)-K-k(S-s))/\mathrm{d} S=0$, namely,
\[
V'(S)=k,
\]
which implies the second part of \eqref{eq:w(s)=w(s)=k}. Similarly,
the starting point should satisfy $V'(s)=k$, which is the first part of \eqref{eq:w(s)=w(s)=k}.
Conditions in \eqref{eq:w(s)=w(s)=k} are called \emph{smooth-pasting conditions} in literature (see e.g., \cite{DaiYao2013a}),
and they ensure that the function $w$ is continuous at $s^*$ and $S^*$.
Boundary conditions \eqref{eq:w-int=K}-\eqref{eq:limw'/h} also can be found in \cite{Bather1966}
to characterize the optimal ordering policy parameters.

Under $\phi=((s^*,S^*),\bm{p}^*)$ policy,
the controlled inventory level process $\{Z_t:t\geq0\}$
is a regenerative process and one cycle can be defined
as the during time between two successive times
when the inventory level $Z$ hits $S^*$.
For the optimal price $\bm{p}^*$, we have
\[
\begin{cases}
p^*(z_1)\leq p^*(z_2) & \text{for $z_1>z_2\geq z^*$,}\\
p^*(z_1)\geq p^*(z_2) & \text{for $z^*\geq z_1>z_2\geq s^*$}.
\end{cases}
\]
The monotonicity can be intuitively explained as follows.
When the inventory level is high, the system manager should set low price to attract more demand,
and should raise price to earn higher profit as the inventory level decreases until it drops to $z^*$,
and then should reduce price to encourage more demand being backlogged,
so that the firm can quickly reach the reorder repoint $s^*$
incurring low shortage cost.
An interesting phenomenon is that a lower price is offered when more demand is backlogged.
This phenomenon can be found in the group-buying market.
For example (see \cite{AnandAron2003}), a group-buying site e.conomy, sells indirect goods and service to customers,
who would be backlogged when placing orders.
The site offers lower price when more customers place their orders before the times when
the site orders goods and service to satisfy the customers.
In addition, notice that in \cite{ChaoZhou2006}'s model with Poisson demand, the highest price is offered at inventory level zero,
where the lowest holding/shortage cost is incurred.
However, in our model with Brownian motion demand, it is offered at a negative inventory level $z^*$.
Intuitively, because of the possibility of negative demand (i.e. customers' return) in Brownian motion demand,
when the highest price is offered at a little negative inventory level,
the possibility that the inventory level fluctuates around level zero would be higher,
and thus this incurs less holding/shortage cost and then brings higher profits.
We believe that similar phenomenon would happen if the customers' return feature is added to the Poisson demand model.

In our paper, $p^*(z)=p_{\pi}(w(z))$ is in a very general formula and it includes both discrete and continuous cases;
see the following examples.
\begin{example}
If
\[
\mu(p)=\frac{\lambda_1}{p+\lambda_0}
\]
with $\lambda_0$ and $\lambda_1$ are positive constants; see \cite{ZhangZhang2012}.
Then we have
\[
\Pi(p,w^*(z))=\lambda_1-\frac{\lambda_1(w^*(z)+\lambda_0)}{p+\lambda_0},
\]
which implies that
\[
p_{\pi}(w^*(z))=
\begin{cases}
\bar{p} & \text{for $w^*(z)\geq -\lambda_0$},\\
\underline{p} & \text{for $w^*(z)<-\lambda_0$}.
\end{cases}
\]
Thus, $w^*(S^*)=k$ in \eqref{eq:w(s)=w(s)=k},
$\lim_{z\to\infty} w^*(z)=-\infty$ in \eqref{eq:limw=},
$\mathrm{d}w^*(z)/\mathrm{d}z<0$ for $z>z^*$ in \eqref{eq:w-z*}
and $z^*<S^*$ in Proposition \ref{prop:w} imply
that there exists a unique $z_{p}\in(S^*,\infty)$ such that
\[
w^*(z_{p})=-\lambda_0,
\]
and then
\[
p^*(z)=
\begin{cases}
\bar{p} & \text{for $z\in[s^*,z_{p}]$},\\
\underline{p} &\text{for $z\in(z_{p},\infty)$}.
\end{cases}
\]
In this example, the optimal pricing is not continuously modified; see Figure \ref{fig:price} (a).

\end{example}

\begin{figure}[tb]
\subfigure[Example 1: discrete price adjustment]
  {\includegraphics[width=7cm]{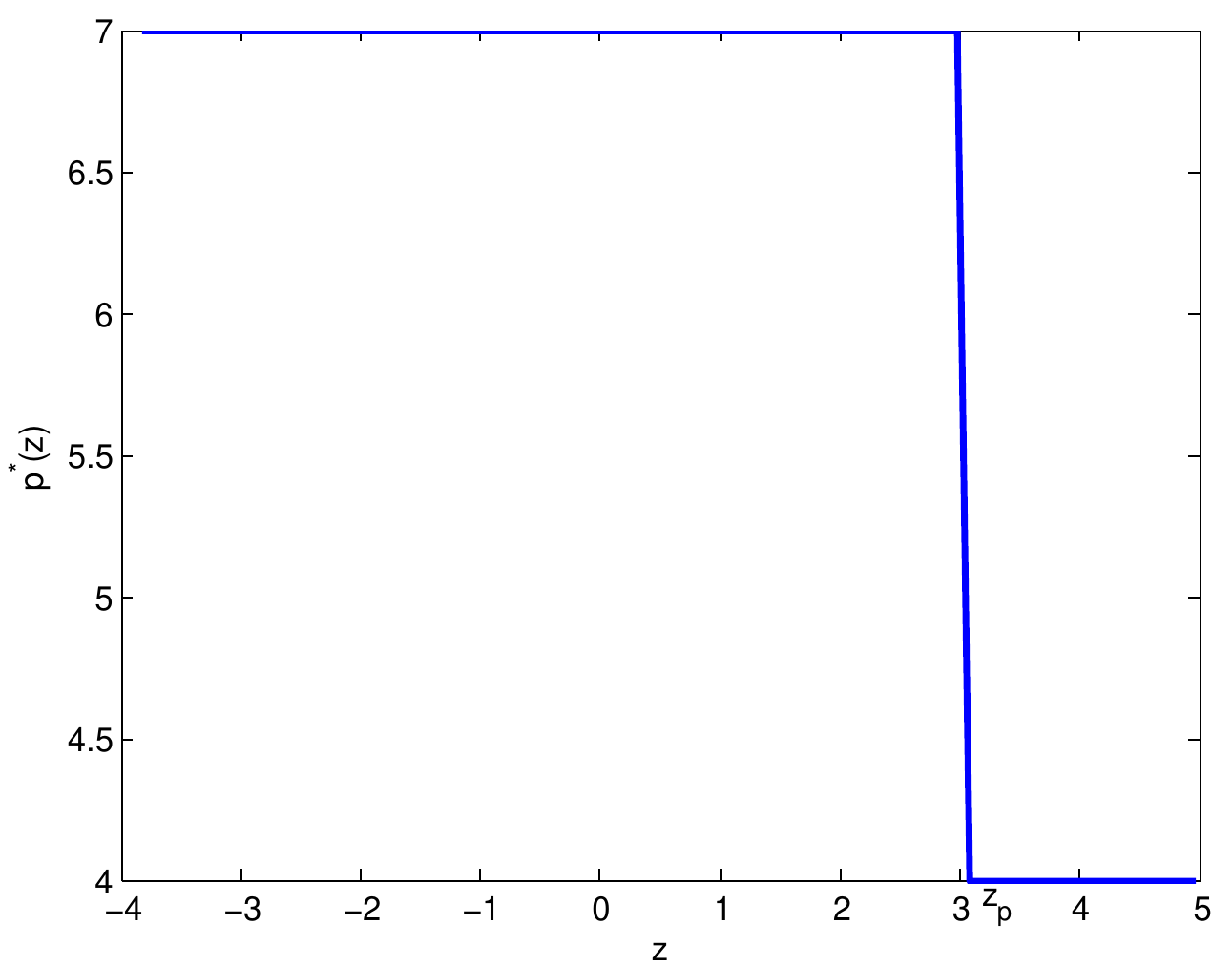}}
\subfigure[Example 2: continuous price adjustment]
   {\includegraphics[width=7cm]{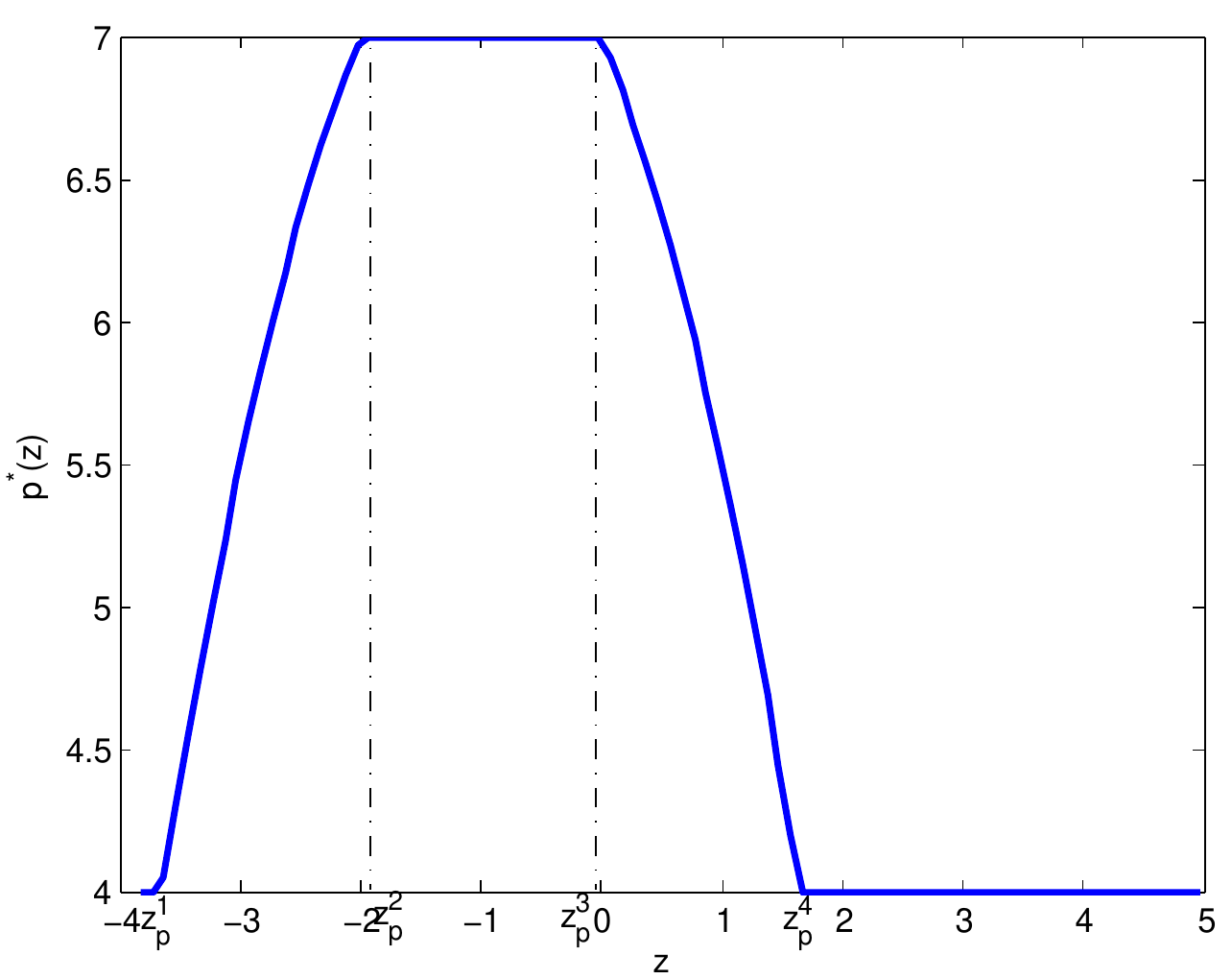}}
\caption{Discrete price adjustment and continuous price adjustment}
\label{fig:price}
\end{figure}

\begin{example}
If
\[
\mu(p)=A-p
\]
with $A>\bar{p}$.
Then we have
\[
\Pi(p,w^*(z))=-p^2+(A+w^*(z))p-Aw^*(z),
\]
which implies that
\[
p_{\pi}(w^*(z))=
\begin{cases}
\bar{p} & \text{for $w^*(z)\geq 2\bar{p}-A$},\\
\frac{A+w^*(z)}{2} & \text{for $2\underline{p}-A<w^*(z)<2\bar{p}-A$},\\
\underline{p} & \text{for $w^*(z)\leq 2\underline{p}-A$}.
\end{cases}
\]
Assume that $2\bar{p}-w^*(z^*)<A<2\underline{p}-k$, then
$w^*(z^*)>2\bar{p}-A$ and $w^*(s^*)=k<2\underline{p}-A$.
Proposition \ref{prop:w} implies that there exist $z_p^i$, $i=1,\cdots,4$,
with $s^*<z_p^1<z_p^2<z^*<z_p^3<z_p^4<S^*$ such that
\[
w^*(z_p^1)=w^*(z_p^4)=2\underline{p}-A\quad \text{and}\quad
w^*(z_p^2)=w^*(z_p^3)=2\bar{p}-A.
\]
Then
\[
p^*(z)=
\begin{cases}
\underline{p} & \text{for $z\in[z_p^4,\infty)$},\\
\frac{A+w^*(z)}{2} & \text{for $z\in(z_p^3,z_p^4)$},\\
\bar{p} & \text{for $z\in[z_p^2,z_p^3]$},\\
\frac{A+w^*(z)}{2} & \text{for $z\in(z_p^1,z_p^2)$},\\
\underline{p} & \text{for $z\in[s^*,z_p^1]$}.
\end{cases}
\]
In this example, the optimal price is continuously modified; see Figure \ref{fig:price} (b).
\end{example}

To prove our main results in Theorem \ref{thm:mainresults}, we use the following approach.
In Section \ref{sec:upperbound}, we establish an upper bound for the average profit
under any admissible policy $\phi\in\Phi$.
Then in Section \ref{sec:pricing}, we study the optimal pricing strategy for the given
$(s,S)$ policy.
Finally, in Section \ref{sec:inventorycontrol}
we show that if there exists an $(s^*,S^*)$ policy with a function $w^*$
and a constant $\gamma^*$ satisfying \eqref{eq:w-ODE}-\eqref{eq:limw'/h},
then $\gamma^*$ is the average profit under $(s^*,S^*)$ inventory control strategy
 with associated optimal pricing strategy $\bm{p}^*$
and it can achieve the upper bound in Section \ref{sec:upperbound}.
Thus $((s^*,S^*),\bm{p}^*)$ is an optimal policy
and $\gamma^*$ is the optimal average profit for our problem \eqref{eq:supAP}.
Then, we complete the proof of Theorem \ref{thm:mainresults}.

\section{Upper bound}
\label{sec:upperbound}

In this section, using the It\^{o} rule, we will establish an upper bound for the average profit
under any admissible policy.
We first need the following lemma.
\begin{lemma}[\textbf{\cite{HeYaoZhang2016} Lemma 3}]
\label{lem:HeYaoZhang}
Let $f(\cdot)\in\mathcal{C}^1(\R)$ and $Z$ be the inventory process
given by \eqref{eq:Z} under an admissible inventory control strategy $\bm{y}$.
Assume that there exist positive number $a_0$, $a_1$ and a positive integer $\ell$
such that
\[
\abs{f'(z)}<a_0+a_1\abs{z}^{\ell} \quad \text{for $z\in\R$}.
\]
Then we have
\begin{eqnarray}
&&\E_x[\abs{f(Z_t)}]<\infty,\label{eq:Ef<}\\
&&\E_x\Big[\int_0^tf'(Z_u)^2\,\mathrm{d}u\Big]<\infty,\label{eq:Eint<}\\
&& \lim_{t\to\infty} \frac{\E_x[\abs{f(Z_t)1_{\{Z_t\geq0\}}}]}{t}=0.\label{eq:lim=0}
\end{eqnarray}
\end{lemma}

Based on Lemma 3.1 above, we derive an upper bound for the average profit
under any admissible policy.
\begin{proposition}
\label{prop:upperbound}
Assume that $h$ satisfies Assumption \ref{assump:h-qd}.
Suppose that there exists a function
$f(\cdot)\in\mathcal{C}^1(\R)$ with an absolutely continuous derivate $f'$,
and a constant $\gamma$ such that the following conditions hold
\begin{align}
&\frac{\sigma^2}{2}f''(z)+\Pi(p,f'(z))-h(z)\leq \gamma\nonumber\\
&\qquad\mbox{for any $z\in\R$ such that $f''$ exists
and for any $p\in[\underline{p},\bar{p}]$},\label{eq:lbPoission}\\
& f(z_1)-f(z_2)\leq K+k\cdot(z_1-z_2)\quad\mbox{for $-\infty<z_2<z_1<\infty$},\label{eq:lbK} \\
& \abs{f'(z)}<
\begin{cases}
a_0 &\text{for $z\leq0$},\\
a_0+a_1 z^{\ell} & \text{for $z\geq0$},
\end{cases}
\label{eq:lbf}
\end{align}
where $\Pi$ is defined in \eqref{eq:Pi},
$a_0$ and $a_1$ are some positive numbers and $\ell$ is some positive integer.
Then we have
\begin{equation}
\label{eq:AP<gamma}
{\sf AP}(x,\phi)\leq \gamma
\end{equation}
for each initial state $x\in \mathbb{R}$ and
each admissible policy $\phi\in \Phi$.
\end{proposition}
\begin{proof}
Using \eqref{eq:Z} and the It\^{o} rule (see e.g., Lemma 3.1 in \cite{DaiYao2013a}), we have
\begin{align*}
f(Z_t)=f(x)+ \int_0^t \big[\frac{1}{2}\sigma^2 f''(Z_u)-\mu(p_u)f'(Z_u)\big]\,\mathrm{d}u
-\sigma\int_0^t f'(Z_u)\,\mathrm{d}B_u
+ \sum_{0\leq u\leq t}\Delta f(Z_u),
\end{align*}
which, together with \eqref{eq:lbPoission}-\eqref{eq:lbK} and $p_t\in[\underline{p},\bar{p}]$ for any $t\geq0$,
implies that
\begin{align}
f(Z_t)\leq f(x)+\gamma t-\int_0^t \big[p_u\mu(p_u)-h(Z_u)\big]\,\mathrm{d}u
-\sigma\int_0^t f'(Z_u)\,\mathrm{d}B_u
+\sum_{i=0}^{N_t}\big(K+k\xi_i\big).\label{eq:f<}
\end{align}
It follows from \eqref{eq:Eint<} and Theorem 3.2.1 in \cite{Oksendal2003} that
\[
\E_x\Big[\int_0^t f'(Z_u)\,\mathrm{d}B_u\Big]=0.
\]
Then taking expectation on both sides of \eqref{eq:f<} and using \eqref{eq:Ef<}, we have
\begin{equation*}
\E_x[f(Z_t)]+\E_x\Big[\int_0^t \big[p_u\mu(p_u)-h(Z_u)\big]\,\mathrm{d}u-\sum_{i=0}^{N_t}\big(K+k\xi_i\big)\Big]
\leq f(x)+\gamma t,
\end{equation*}
which yields
\[
\limsup_{t\to\infty} \frac{\E_x[f(Z_t)]}{t} + \AP(x,\phi)\leq\gamma.
\]
Thus, \eqref{eq:AP<gamma} holds when
\[
\limsup_{t\to\infty} \frac{\E_x[f(Z_t)]}{t}\geq 0.
\]
Otherwise, there exists a constant $\delta>0$ such that
\begin{equation*}
\limsup_{t\to\infty} \frac{\E_x[f(Z_t)]}{t}<-\delta,
\end{equation*}
which, together with \eqref{eq:lim=0}, implies that
\[
\limsup_{t\to\infty} \frac{\E_x[f(Z_t)1_{\{Z_t<0\}}]}{t}<-\delta.
\]
Thus, we have
\[
\E_x[f(Z_t)1_{\{Z_t<0\}}]<-\frac{\delta t}{2}
\]
for $t$ sufficiently large. It follows from \eqref{eq:lbf} that
there exists some constant $a_2>0$ such that
$f(z)\geq a_0z-a_2$ for $z<0$.
Then, we must have
\[
\E_x[\abs{Z_t}]\geq \frac{\delta t-2a_2}{-2a_0},
\]
which, together with \eqref{eq:h-p-q}, implies that
\begin{align*}
\AP(x,\phi)&\leq \liminf_{t\to\infty}\frac{1}{t}\E_x
\Big[\int_0^t \big[p_u\mu(p_u)-h(Z_u)\big]\,\mathrm{d}u\Big]\\
&\leq\kappa-\liminf_{t\to\infty}\frac{1}{t}\E_x \Big[\int_0^t h(Z_u)\,\mathrm{d}u\Big]\\
&\leq\kappa-d_1\liminf_{t\to\infty}\frac{1}{t}\E_x\Big[\int_0^t \abs{Z_u}\,\mathrm{d}u\Big]-d_2\\
&=-\infty,
\end{align*}
where $\kappa\triangleq\max_{p\in[\underline{p},\bar{p}]}p\mu(p)<\infty$.
Hence, we must have \eqref{eq:AP<gamma}.
\end{proof}

\section{Optimal pricing strategy under the given $(s,S)$ policy}
\label{sec:pricing}

In this section, we analyze that the optimal pricing strategy under the given $(s,S)$ policy
can be characterized by \eqref{eq:pricing}.
Specifically, under the given $(s,S)$ policy, it follows from \eqref{eq:Z}, \eqref{eq:sS-0} and \eqref{eq:sS-i} that
the inventory level $Z$ will be derived as
\begin{equation*}
\label{eq:Z-sS}
Z_t=x-\int_0^t \mu(p_u)\,\mathrm{d}u -\sigma B_t+(S-x)1_{\{x\leq s\}}+N_t(S-s),\quad t\geq0,
\end{equation*}
and then the pricing strategy $\bm{p}_{(s,S)}=\{p_t=p_{\pi}(w_{(s,S)}(Z_t)):t\geq0\}$ in \eqref{eq:pricing} is defined well for given function $w$.
Next, we first, in Theorem \ref{thm:sSpolicy}, prove the optimality of pricing strategy defined in \eqref{eq:pricing} under the given $(s,S)$ policy,
and then in Lemma \ref{lemma:price}, we give some useful properties
for this class of pricing strategy.

\begin{theorem}
\label{thm:sSpolicy}
Under Assumption \ref{assumption-mu},
given $(s,S)$ policy, if there exists $(w_{(s,S)}(\cdot),\gamma_{(s,S)})$,
where $w_{(s,S)}(z)$ is continuously differentiable in $\R$ and twice continuously differentiable in $\R\setminus \{0\}$
and  $\gamma_{(s,S)}$ is a constant,
satisfying ODE \eqref{eq:w-ODE}
with boundary conditions \eqref{eq:w-int=K} and \eqref{eq:limw'/h}.
then $\bm{p}_{(s,S)}$ is the optimal pricing strategy
and $\gamma_{(s,S)}$ is the associated long-run average profit.
\end{theorem}
\begin{proof}
We first prove that $\gamma_{(s,S)}$ is the average profit under $(s,S)$ policy and
pricing strategy $\bm{p}_{(s,S)}$.
Define
\[
V(z)=\int_{s}^{z}w_{(s,S)}(u)\,\mathrm{d}u.
\]
Thus, \eqref{eq:w-ODE}, \eqref{eq:w-int=K} and \eqref{eq:limw'/h} imply that $V$ satisfies
\begin{align}
\label{eq:V-ODE}
\frac{\sigma^2}{2}V''(z)+\pi(V'(z))-h(z)=\gamma_{(s,S)} \quad\text{for $z\in\R$}
\end{align}
with boundary conditions
\begin{align}
&V(S)-V(s)=K+k\cdot(S-s),\label{eq:V-int=K}\\
&\lim_{z\to\infty}V''(z)/h(z)=0.\label{eq:limV/z}
\end{align}
Using the It\^{o} rule, we have
\begin{align*}
V(Z_t)=V(Z_0)+ \int_0^t \big[\frac{1}{2}\sigma^2V''(Z_u)-\mu(p_{\pi}(V'(Z_u)))V'(Z_u)\big]\,\mathrm{d}u
-\sigma\int_0^tV'(Z_u)\,\mathrm{d}B_u+ \sum_{0< u\leq t}\Delta V(Z_u),
\end{align*}
which, together with \eqref{eq:pi}, \eqref{eq:V-ODE} and \eqref{eq:V-int=K}, implies that
\begin{align}
V(Z_t)=&V(Z_0)+ \gamma_{(s,S)} t-\int_0^t \big[p_{\pi}(V'(Z_u))\cdot \mu(p_{\pi}(V'(Z_u)))
-h(Z_u)\big]\,\mathrm{d}u \nonumber\\
&-\sigma\int_0^tV'(Z_u)\,\mathrm{d}B_u
+\sum_{i=1}^{N_t}\big(K+k(S-s)\big).\label{eq:V(Z(t))=}
\end{align}
Furthermore, \eqref{eq:limV/z} and Assumption \ref{assump:h-qd} (d)
imply that $V$ satisfies the conditions in Lemma \ref{lem:HeYaoZhang},
thus we have
\[
\E_x[\abs{V(Z_t)}]<\infty\quad\text{and}\quad
\E_x\Big[\int_0^tV'(Z_u)^2\,\mathrm{d}u\Big]<\infty.
\]
Therefore, taking expectation on both sides of \eqref{eq:V(Z(t))=}, we have
\[
\AP(x,((s,S),\bm{p}_{(s,S)}))=\gamma_{(s,S)},
\]
i.e., $\gamma_{(s,S)}$ is the long-run average profit under
$(s,S)$ policy and pricing strategy $\bm{p}_{(s,S)}$.

It remains to prove that $\bm{p}_{(s,S)}$ is the optimal pricing strategy
under the given $(s,S)$ policy, i.e.,
\begin{align}
\label{eq:AP<gamma-sS}
\AP(x,((s,S),\bm{p}))\leq \gamma_{(s,S)} \quad \text{for any $\bm{p}$}.
\end{align}
Given any admissible pricing strategy $\bm{p}=\{p_t:t\geq0\}$,
it follows from the definition of $\pi$ in \eqref{eq:pi} that
\begin{align}
\frac{\sigma^2}{2}V''(z)-\mu(p)V'(z)+p\mu(p)-h(z)
\leq\frac{\sigma^2}{2}V''(z)+\pi(V'(z))-h(z)
=\gamma_{(s,S)}.\label{eq:gammasS-inequality}
\end{align}
Thus under any admissible pricing control $\bm{p}=\{p_t:t\geq0\}$ and using It\^{o} rule,
we have
\begin{align*}
V(Z_t)=V(Z_0)+ \int_0^t \big[\frac{1}{2}\sigma^2V''(Z_u)-\mu(p_s)V'(Z_u)\big]\,\mathrm{d}u
-\sigma\int_0^tV'(Z_u)\,\mathrm{d}B_u
+ \sum_{0< u\leq t}\Delta V(Z_u).
\end{align*}
Then it follows from \eqref{eq:gammasS-inequality} that
\begin{align*}
V(Z_t)\leq V(Z_0)+ \gamma_{(s,S)} t-\int_0^t \big[p_u\cdot \mu(p_u)
-h(Z_u)\big]\,\mathrm{d}u
-\sigma\int_0^tV'(Z_u)\,\mathrm{d}B_u
+\sum_{i=1}^{N_t}\big(K+k(S-s)\big).
\end{align*}
The following proof for \eqref{eq:AP<gamma-sS} will be same to that in the first paragraph except
$\leq$ instead of $=$ and thus it is omitted.
\end{proof}

The following lemma states some properties about the pricing strategy defined in \eqref{eq:pricing},
that will be used in next section. See Appendix \ref{appendix} for proof details.
\begin{lemma}
\label{lemma:price}
Under Assumption \ref{assumption-mu}, we have that
\begin{itemize}
\item[(1)] $p_{\pi}(w)$ is increasing in $w\in\R$.
\item[(2)] $\pi(w)$ is continuously differentiable and
\begin{equation}
\label{eq:pi'}
\pi'(w)< 0.
\end{equation}
Furthermore,
\begin{equation}
\label{eq:limpi}
\lim_{w\to-\infty}\pi(w)=+\infty\quad\text{and}\quad
\lim_{w\to\infty}\pi(w)=-\infty.
\end{equation}
\end{itemize}
\end{lemma}

\section{Optimal policy}
\label{sec:inventorycontrol}

In this section, we provide the proof for our main results, Theorem \ref{thm:mainresults}.
Before doing that, we first prove some important properties of $w^*$ satisfying
ODE \eqref{eq:w-ODE} with boundary conditions \eqref{eq:w-int=K}-\eqref{eq:limw'/h}.

\begin{proposition}
\label{prop:w}
Under Assumption \ref{assumption-mu} and \ref{assump:h-qd},
if there exists $((s^*,S^*),w^*(\cdot),\gamma^*)$ satisfying the assumptions in Theorem \ref{thm:mainresults},
then we have that
\begin{equation}
\label{eq:limw=}
\lim_{z\to\infty}w^*(z)=-\infty.
\end{equation}
and there exists a unique $z^*$ with $z^*\leq0$ and $s^*<z^*<S^*$ such that
\begin{equation}
\label{eq:w-z*}
\frac{\mathrm{d}w^*(z)}{\mathrm{d}z}
\left\{
\begin{array}{ll}
>0 & \text{for $z<z^*$},\\
=0 & \text{for $z=z^*$},\\
<0 & \text{for $z>z^*$}.
\end{array}
\right.
\end{equation}
\end{proposition}
\begin{proof}
 We first prove \eqref{eq:limw=}.
It follows from \eqref{eq:w-ODE} that
\[
\frac{\frac{\sigma^2}{2}\frac{\mathrm{d}w^*(z)}{\mathrm{d}z}+\pi(w^*(z))}{h(z)}=1+\frac{\gamma^*}{h(z)},
\]
which, together with \eqref{eq:limh}, implies that
\begin{equation*}
\lim_{z\to\infty}\frac{\frac{\sigma^2}{2}\frac{\mathrm{d}w^*(z)}{\mathrm{d}z}+\pi(w^*(z))}{h(z)}=1.
\end{equation*}
Then, \eqref{eq:limw'/h} implies that
\[
\lim_{z\to\infty}\frac{\pi(w^*(z))}{h(z)}=1,
\]
which, together with \eqref{eq:limh} and \eqref{eq:limpi}, implies \eqref{eq:limw=}.

We prove \eqref{eq:w-z*} in two steps:
\begin{itemize}
\item[(a)]
\begin{equation}
\label{eq:dw(z)/dz<0}
\frac{\mathrm{d}w^*(z)}{\mathrm{d}z}<0\quad \text{for $z\in(0,\infty)$.}
\end{equation}
\item[(b)] The exists a unique $z^*$ with $z^*\leq0$ such that
\begin{equation}
\label{eq:z*-2}
\frac{\mathrm{d}w^*(z)}{\mathrm{d}z}
\left\{
\begin{array}{ll}
>0 & \text{for $z\in(-\infty,z^*)$},\\
=0 & \text{for $z=z^*$},\\
<0 & \text{for $z\in(z^*,0]$}.
\end{array}
\right.
\end{equation}
\end{itemize}
We next give the details step by step.

(a) Assume that there exists a $z_1>0$ such that
\begin{equation}
\label{eq:dw(z1)/dz}
\frac{\mathrm{d}w^*(z_1)}{\mathrm{d}z}\geq0.
\end{equation}
First, \eqref{eq:w-ODE} implies that
\[
\frac{\sigma^2}{2}\frac{\mathrm{d}^2w^*(z_1)}{\mathrm{d}z^2}
+\pi'(w^*(z_1))\frac{\mathrm{d}w^*(z_1)}{\mathrm{d}z}-h'(z_1)=0,
\]
which, together with \eqref{eq:h'}, \eqref{eq:pi'} and \eqref{eq:dw(z1)/dz},
implies that
\[
\frac{\mathrm{d}^2w^*(z_1)}{\mathrm{d}z^2}>0.
\]
Then, the continuity of $\mathrm{d}^2w^*(z)/\mathrm{d}z^2$ in $z\in(0,\infty)$ implies that
there must exist $z_2\in(z_1,\infty)$
such that $\mathrm{d}^2w^*(z)/\mathrm{d}z^2>0$ for $z\in[z_1,z_2]$,
and then \eqref{eq:dw(z1)/dz} implies that $\mathrm{d}w^*(z)/\mathrm{d}z>0$ for $z\in(z_1,z_2]$.
Thus, the continuity of $\mathrm{d}w^*(z)/\mathrm{d}z$ and \eqref{eq:limw=} imply that
 there exists $z_3$ with $z_3>z_2$ such that
\begin{equation}
\label{eq:dw(z3)/dz}
\frac{\mathrm{d}w^*(z_3)}{\mathrm{d}z}=0\quad\text{and}\quad w^*(z_3)>w^*(z_1),
\end{equation}
which plus \eqref{eq:pi'} imply that
\begin{equation}
\label{eq:pi(w(z3))<pi(w(z1))}
\pi(w^*(z_3))<\pi(w^*(z_1)).
\end{equation}
Then we have
\[
\gamma^*
=\frac{\sigma^2}{2}\frac{\mathrm{d}w^*(z_3)}{\mathrm{d}z}+\pi(w^*(z_3))-h(z_3)
<\frac{\sigma^2}{2}\frac{\mathrm{d}w^*(z_1)}{\mathrm{d}z}+\pi(w^*(z_1))-h(z_1)
=\gamma^*,
\]
where the inequality follows from \eqref{eq:dw(z1)/dz}, \eqref{eq:dw(z3)/dz}, \eqref{eq:pi(w(z3))<pi(w(z1))}
and $h(z_3)>h(z_1)$. This contradicts with the fact $\gamma^*=\gamma^*$,
and thus we have \eqref{eq:dw(z)/dz<0}.

(b) Define
\begin{equation}
\label{eq:z*}
z^*=\argmax_{z\in[s^*,S^*]} w^*(z).
\end{equation}
it follows from \eqref{eq:dw(z)/dz<0} that $z^*\leq 0$.
We next show the uniqueness of $z^*$.
Assume that there exist $z_1^*$ and $z_2^*$ with $z_1^*<z_2^*\leq 0$ satisfying
\eqref{eq:z*}, we have
\begin{equation}
\label{eq:w(z1*)=w(z2*)}
w(z_1^*)=w(z_2^*)\quad \text{and}\quad
\frac{\mathrm{d} w^*(z_i^*)}{\mathrm{d} z}=0\quad \text{for $i=1,2$},
\end{equation}
which, together with \eqref{eq:w-ODE},
and $h(z_1^*)>h(z_2^*)$ for $z_1^*<z_2^*\leq0$, imply that
\[
\gamma^*=\frac{\sigma^2}{2}\frac{\mathrm{d}w^*(z_1^*)}{\mathrm{d} z}+\pi(w^*(z_1^*))-h(z_1^*)
<\frac{\sigma^2}{2}\frac{\mathrm{d}w^*(z_2^*)}{\mathrm{d} z}+\pi(w^*(z_2^*))-h(z_2^*)
=\gamma^*.
\]
This contradiction implies that $z^*$ defined in \eqref{eq:z*} is unique.
Furthermore, it follows from  \eqref{eq:w-int=K} and \eqref{eq:w(s)=w(s)=k}
that $z^*\in(s^*,S^*)$.

To prove \eqref{eq:z*-2}, we first prove $\mathrm{d}w^*(z)/\mathrm{d} z>0$ for $z\in(-\infty,z^*)$.
Assume that there exists $z_4\in(-\infty,z^*)$ such that
\begin{equation}
\label{eq:w'(z4)<0}
\frac{\mathrm{d}w^*(z_4)}{\mathrm{d}z}\leq 0.
\end{equation}
Since $z^*$ is the unique maximum point for $w^*$ in $[s^*,S^*]$ and $z^*\in(s^*,S^*)$, there exists $z_5\in(z_4,z^*)$ such that
\begin{equation}
\label{eq:w'(z5)>0}
\frac{\mathrm{d}w^*(z_5)}{\mathrm{d}z}>0.
\end{equation}
Thus, the continuity of $\mathrm{d}w^*(z)/\mathrm{d}z$, \eqref{eq:w'(z4)<0} and \eqref{eq:w'(z5)>0} imply that
\begin{equation}
\label{eq:z6}
z_6=\min\{z\in[z_4,z_5]:\frac{\mathrm{d}w^*(z)}{\mathrm{d}z}=0\}
\end{equation}
is well defined.
Furthermore, it follows from \eqref{eq:w-ODE} that
\[
\frac{\sigma^2}{2}\frac{\mathrm{d}^2w^*(z_4)}{\mathrm{d}z^2}
+\pi'(w^*(z_4))\frac{\mathrm{d}w^*(z_4)}{\mathrm{d}z}-h'(z_4)=0
\]
which, together with \eqref{eq:h'}, \eqref{eq:pi'} and \eqref{eq:w'(z4)<0},
implies that
\[
\frac{\mathrm{d}^2w^*(z_4)}{\mathrm{d}z^2}<0.
\]
Thus, the continuity of $\mathrm{d}^2w^*(z)/\mathrm{d}z^2$ in $z\in(-\infty,0)$ implies that
there exists $z_7$ with $z_7<z_4$ such that
$\mathrm{d}^2w^*(z)/\mathrm{d}z^2<0$ for $z\in[z_7,z_4]$, which plus \eqref{eq:w'(z4)<0}
yield that $\mathrm{d}w^*(z)/\mathrm{d}z<0$ for $z\in[z_7,z_4)$.
It follows from the definition of $z_6$ in \eqref{eq:z6} that
\begin{equation}
\label{eq:w'(z)<0}
\frac{\mathrm{d}w^*(z)}{\mathrm{d}z}<0=\frac{\mathrm{d}w^*(z_6)}{\mathrm{d}z}\quad \text{for $z\in[z_7,z_6)$}
\end{equation}
and then
\begin{equation}
\label{eq:w*(z7)>w*(z6)}
w^*(z_7)>w^*(z_6).
\end{equation}
Therefore, \eqref{eq:pi'}, \eqref{eq:w'(z)<0}, \eqref{eq:w*(z7)>w*(z6)} and $h(z_6)<h(z_7)$ imply that
\[
\gamma^*=\frac{\sigma^2}{2}\frac{\mathrm{d}w^*(z_6)}{\mathrm{d}z}
+\pi(w^*(z_6))-h(z_6)>\frac{\sigma^2}{2}\frac{\mathrm{d}w^*(z_7)}{\mathrm{d}z}+\pi(w^*(z_7))-h(z_7)=\gamma^*.
\]
This contradicts with the fact $\gamma^*=\gamma^*$.

Finally, we prove $\mathrm{d}w^*(z)/\mathrm{d}z<0$ for $z\in(z^*,0]$.
If $z^*=0$, $(z^*,0]=\emptyset$ and the result immediately holds.
Thus we only need to prove the case when $z^*<0$.
Assume that there exists a $z_8\in(z^*,0]$ with
\[
\frac{\mathrm{d}w^*(z_8)}{\mathrm{d}z}\geq0.
\]
Let
\begin{equation*}
z_9=\min\{z\in(z^*,z_8]:\frac{\mathrm{d}w^*(z)}{\mathrm{d}z}=0\}.
\end{equation*}
The uniqueness of $z^*$ in $[s^*,S^*]$, $z^*\in(s^*,S^*)$
and the continuity of $\mathrm{d}w^*(z)/\mathrm{d}z$
imply that $z_9$ is well defined and
\begin{equation}
\label{eq:w(z*)>w(z9)}
w^*(z^*)>w^*(z_9).
\end{equation}
Therefore, \eqref{eq:pi'}, \eqref{eq:w(z*)>w(z9)},
$\mathrm{d}w^*(z^*)/\mathrm{d}z=\mathrm{d}w^*(z_9)/\mathrm{d}z=0$  and $h(z^*)>h(z_9)$ imply that
\[
\gamma^*=\frac{\sigma^2}{2}\frac{\mathrm{d}w^*(z^*)}{\mathrm{d}z}
+\pi(w^*(z^*))-h(z^*)<\frac{\sigma^2}{2}\frac{\mathrm{d}w^*(z_9)}{\mathrm{d}z}+\pi(w^*(z_9))-h(z_9)=\gamma^*,
\]
which contradicts with $\gamma^*=\gamma^*$.
\end{proof}

We are ready to prove our main results.
\begin{proof}[\textbf{Proof of Theorem \ref{thm:mainresults}}]
(1) Since $(w^*,\gamma^*)$ satisfies \eqref{eq:w-ODE} with boundary conditions \eqref{eq:w-int=K}
and \eqref{eq:limw'/h},
Theorem \ref{thm:sSpolicy} tells us that $\gamma^*=\gamma_{(s^*,S^*)}$,
$\bm{p}^*=\bm{p}_{(s^*,S^*)}$ and
$\gamma^*$ is the average profit under $((s^*,S^*),\bm{p}^*)$.
Let
\begin{equation}
\label{eq:V*}
V^*(z)=
\begin{cases}
\int_{s^*}^{z}w^*(y)\,\mathrm{d}y & \text{for $z\in[s^*,\infty)$},\\
k\cdot (x-s^*) & \text{for $z\in(-\infty,s^*)$}.
\end{cases}
\end{equation}
If we can show that $V^*$ and $\gamma^*$ satisfy all conditions
in Proposition \ref{prop:upperbound}, then $\gamma^*$ is an upper bound
for the average profit under any admissible policy $\phi$.
Thus, the optimality of $((s^*,S^*),\bm{p}^*)$ is proven.
We next check that $V^*$ and $\gamma^*$ satisfy all conditions
in Proposition \ref{prop:upperbound}.

First, it follows from \eqref{eq:V*} and \eqref{eq:w(s)=w(s)=k} that $V^*\in\mathcal{C}^1(\R)$ and
is twice continuously differentiable in $\R\setminus\{s^*\}$.
Thus, we have that $V^*(\cdot)\in\mathcal{C}^1(\R)$ with an absolutely continuous derivative
$\mathrm{d}V^*(z)/{\sf d}z$.

We next check \eqref{eq:lbPoission}.
It follows from \eqref{eq:w-ODE} and \eqref{eq:V*} that
for $z\in(s^*,\infty)$,
\[
\frac{\sigma^2}{2}\frac{\mathrm{d}^2V^*(z)}{\mathrm{d}z^2}+
\Pi(p,\frac{\mathrm{d}V^*(z)}{\mathrm{d}z})-h(z)\leq
\frac{\sigma^2}{2}\frac{\mathrm{d}^2V^*(z)}{\mathrm{d}z^2}+
\pi(\frac{\mathrm{d}V^*(z)}{\mathrm{d}z})-h(z)=\gamma^*.
\]
For $z\in(-\infty,s^*)$, we have
\begin{align*}
\frac{\sigma^2}{2}\frac{\mathrm{d}^2V^*(z)}{\mathrm{d}z^2}+
\Pi(p,\frac{\mathrm{d}V^*(z)}{\mathrm{d}z})-h(z)
&\leq\frac{\sigma^2}{2}\frac{\mathrm{d}^2V^*(z)}{\mathrm{d}z^2}+
\pi(\frac{\mathrm{d}V^*(z)}{\mathrm{d}z})-h(z)\\
&=\pi(k)-h(z)\\
&<\frac{\sigma^2}{2}\frac{\mathrm{d}w^*(s^*)}{\mathrm{d}z}+\pi(w^*(s^*))-h(s^*)\\
&=\gamma^*,
\end{align*}
where the first inequality follows from the definition of $\pi$ in \eqref{eq:pi},
the first equality follows from \eqref{eq:V*},
and the second inequality follows from $\mathrm{d}w^*(s^*)/\mathrm{d}z>0$ (see \eqref{eq:w-z*})
and $h(z)>h(s^*)$ (see \eqref{eq:h'} and $z<s^*<z^*\leq 0$), and the last equality follows from \eqref{eq:w-ODE}.

Further, we check \eqref{eq:lbK}. It follows from \eqref{eq:w(s)=w(s)=k}, \eqref{eq:w-z*} and \eqref{eq:V*}
that
\begin{equation}
\label{eq:dV*/dz}
\frac{\mathrm{d}V^*(z)}{\mathrm{d}z}
\begin{cases}
>k & \text{for $z\in(s^*,S^*)$},\\
\leq k & \text{for $z\not\in(s^*,S^*)$}.
\end{cases}
\end{equation}
Then, for $-\infty<z_2<z_1<\infty$, we have
\begin{align*}
V^*(z_1)-V^*(z_2)-k\cdot(z_1-z_2)
&=\int_{z_2}^{z_1}\big[\frac{\mathrm{d}V^*(y)}{\mathrm{d}z}-k\big]\,{\sf d}y\\
&\leq \int_{s^*}^{S^*}\big[\frac{\mathrm{d}V^*(y)}{\mathrm{d}z}-k\big]\,{\sf d}y\\
&=K,
\end{align*}
where the inequality follows from \eqref{eq:dV*/dz} and the last equality follows from \eqref{eq:w-int=K}.

Finally, we check \eqref{eq:lbf}.
Assumption \ref{assump:h-qd} (d) and \eqref{eq:limw'/h} imply that there exists an positive integer
$n$ such that
\[
\lim_{z\to\infty}\frac{\mathrm{d}^2V^*(z)/\mathrm{d}z^2}{z^n}=0,
\]
which yields that there exist some positive numbers $a_1$ and $a_2$ such that
\[
\frac{\mathrm{d}V^*(z)}{\mathrm{d}z}<a_2+a_1 z^{n+1}\quad \text{for $z\geq0$}.
\]
Define $a_0=\max\{k,a_2,\max_{z\in[s^*,0]}\frac{\mathrm{d}V^*(z)}{\mathrm{d}z}\}$, and then we have \eqref{eq:lbf}.

(2) Since $p^*(z)=p_{\pi}(w^*(z))$, this part can be immediately got by Lemma \ref{lemma:price} and Proposition \ref{prop:w}.
\end{proof}

\section{Conclusions}
\label{sec:conclusion}

In this paper, we have studied a Brownian inventory model where product's price and
ordering decisions are simultaneously determined.
The demand drift rate is price-dependent and excess demand is backlogged.
There are two costs: holding/shortage cost and ordering cost.
Our objective is to find a pricing strategy and an ordering strategy
to maximize the expected average profit.
Using an upper bound approach, we have shown that an $(s^*,S^*,\bm{p}^*)$ policy
is optimal and that the optimal price $p^*(z)$
first increases and then decreases as the inventory level $z$ decreases from $S^*$ to $s^*$.

\section*{Acknowledgements}

The author would like to thank the department editor, the associate editor and the anonymous referees for their thoughtful
comments and suggestions, which have helped to significantly improve this paper.
The author thanks Ping Cao from University of Science and Technology of China, Xin Chen from University of Illinois at Urbana-Champaign, 
Shuangchi He from National University of Singapore,
and Guodong Pang from Pennsylvania State University for their useful comments.

\section*{Funding}
The work was supported in part by the National
Natural Science Foundation of China under grant 11401566.

\section*{Notes on contributors}
Dacheng Yao is an Associate Professor of Academy of Mathematics and Systems Science, Chinese Academy of Sciences.
His research interests include inventory control, stochastic models and applied probability.
He received the B.Sc. degree in information and computing science from the Shandong University in 2005,
and the Ph.D. degree in operational research and cybernetics from the Chinese Academy of Sciences in 2010.
He held a postdoctoral position at the Chinese Academy of Sciences from 2010 to 2012. He had visiting positions
at the Georgia Institute of Technology, the National University of Singapore,
and the Chinese University of Hong Kong.

\appendix

\section{Proof of Lemma \ref{lemma:price}}
\label{appendix}

(1) It follows from Assumption \ref{assumption-mu} (b) that
\[
\frac{\partial^2 \Pi(p,w)}{\partial p\partial w}=-\mu'(p)>0,
\]
which tells us that $\Pi(p,w)$ is supermodular,
and then Theorem 2.3.7 in \cite{Simchi-LeviChenBramel2005}
implies that $p_{\pi}(w)$ is increasing in $w\in\R$.

(2) By Assumption \ref{assumption-mu} (a), we have that
$\Pi(p,w)$ is twice continuously differentiable,
and then $\pi(w)$ is continuously differentiable.
We next prove \eqref{eq:pi'}. To prove it, we will show that
\begin{align}
\label{eq:pi'=}
\pi'(w)=-\mu(p_{\pi}(w)),
\end{align}
which yields \eqref{eq:pi'}.
If $p_{\pi}(w)=\underline{p}$ or $p_{\pi}(w)=\bar{p}$,
we immediately get \eqref{eq:pi'=}.
Otherwise, we have $p_{\pi}(w)\in(\underline{p},\bar{p})$,
and then $p_{\pi}(w)$ satisfies
\[
\frac{\partial \Pi(p_{\pi}(w),w)}{\partial p}
=0.
\]
Thus, we have
\[
\pi'(w)=\frac{\partial \Pi(p_{\pi}(w),w)}{\partial p}\frac{{\sf d}p_{\pi}(w)}{{\sf d}w}
+\frac{\partial \Pi(p_{\pi}(w),w)}{\partial w}
=\frac{\partial \Pi(p_{\pi}(w),w)}{\partial w}
=-\mu(p_{\pi}(w)),
\]
which yields \eqref{eq:pi'=}.

It remains to prove \eqref{eq:limpi}.
For them, we have
\begin{align*}
&\lim_{w\to-\infty}\pi(w)
\geq\lim_{w\to-\infty}\mu(\underline{p})[\underline{p}-w]=+\infty, \quad\text{and}\\
&\lim_{w\to\infty}\pi(w)\leq\lim_{w\to\infty}
\big[\max_{p\in[\underline{p},\bar{p}]}\mu(p)p-\mu(\bar{p})w\big]=-\infty,
\end{align*}
where the first inequality follows from the definition of $\pi$,
and the last inequality follows from that for any $w\geq0$,
$\pi(w)=\mu(p_{\pi}(w))[p_{\pi}(w)-w]\leq \max_{p\in[\underline{p},\bar{p}]}\mu(p)p-
\min_{p\in[\underline{p},\bar{p}]}\mu(p)w
=\max_{p\in[\underline{p},\bar{p}]}\mu(p)p-
\mu(\bar{p})w$.
\qed

\end{document}